\documentclass[10pt, reqno]{amsart}
\usepackage{amssymb,latexsym,amsmath,amsfonts}
\usepackage{latexsym}
\usepackage{latexsym}
\usepackage{hyperref}
\usepackage[mathscr]{eucal}
\usepackage[english]{babel}
\newtheorem{definition}{Definition}
\newtheorem{thm}{Theorem}
\newtheorem{lem}{ \bf Lemma}

\newtheorem{cor}{\bf Corollary}
\newtheorem{rem}{\bf Remark}
\newcommand{\be}{\begin{equation}}
\newcommand{\ee}{\end{equation}}
\newcommand{\Bea}{\begin{eqnarray*}}
\newcommand{\Eea}{\end{eqnarray*}}
\newcommand{\bea}{\begin{eqnarray}}
\newcommand{\eea}{\end{eqnarray}}

\numberwithin{equation}{section}

\def\dg{{\delta}_g}

\def\Rt{{\check{R}}}
\def\la{\Delta}

\def\Rc{\stackrel{\circ}{R}}

\begin{document}
\title[Stability of Quadratic curvature Functionals]{Stability of Quadratic curvature Functionals at product of Einstein manifolds}

\author{Atreyee Bhattacharya and Soma Maity}
\address{Indian Institute of Science Education and Research, Bhopal, India}
\email{atreyee@iiserb.ac.in}
\address{Indian Institute of Science Education and Research, Mohali, India}
\email{soma123maity@gmail.com}


\begin{abstract}
Consider Riemannian functionals defined by $L^2$-norms of Ricci curvature, scalar curvature, Weyl curvature and Riemannian curvature. Rigidity, stability and local minimizing properties of Eistein metrics as critical metrics of these quadratic functionals have been studied in \cite{GV}. In this paper, we study the same for products of Einstein metrics with Einstein constants of possibly opposite signs. In particular, we prove that the product of a spherical space form and a compact hyperbolic manifold is unstable for certain quadratic functionals if the first eigenvalue of the Laplacian of the hyperbolic manifold is sufficiently small. We also prove the stability of $L^{\frac{n}{2}}$-norm of Weyl curvature at compact quotients of $S^n\times \mathbb{H}^m$.
\end{abstract}


\maketitle
\section{Introduction}
Let $M^n$ be a closed smooth manifold of dimension $n\geq 3$. We consider quadratic Riemannian functionals $\mathcal{R},$ $\mathcal{R}ic,$ $\mathcal{S}$ and $\mathcal{W}_2$ defined on the space of Riemannian metrics on $M$ by
\Bea
\mathcal{R}(g)= \int_M |R_g|^2 dv_g,\ \ \mathcal{R}ic(g)= \int_M |r_g|^2 dv_g, \ \ \mathcal{S}(g)= \int_M s_g^2 dv_g, \ \ \mathcal{W}_2(g)=\int_M |W|^2dv_g 
\Eea
where $R_g$, $r_g$, $s_g$ and $W$ denote the Riemannian curvature, Ricci curvature , scalar curvature and Weyl curvature of the Riemannian manifold $(M,g)$ respectively. If $(M,g)$ is not the standard sphere,  the space of symmetric two tensors $S^2M$ on $M$ admits the following orthogonal decomposition (see Lemma (4.57) in \cite{BE}).
\be
 \label{eqn:decomposition} S^2M=Im\dg^*\oplus C^{\infty}Mg\oplus (\dg^{-1}(0)\cap tr^{-1}(0))
\ee
As the Riemannian functionals above are invariant under the action of the diffeomorphism group of $M$ on the space of Riemannian metrics, gradients of these functionals have no component along $Im \dg^*.$ Also, since these functionals are not scale invariant, we restrict them to the space of Riemannian metrics with constant volume. 
\begin{definition}
Let $(M,g)$ be a closed Riemannian manifold which is not the standard sphere. If $g$ is a critical metric of a Riemannian functional $\mathcal{F}$ then $g$ is called stable for $\mathcal{F}$ if there exists a constant $C>0$ such that 
\be \label{eqn:hessian}
H_F(h,h)\geq C\|h\|^2 \ \ {\rm for \ all} \ h\in \{fg: f\in C^{\infty}M \text{ and }\int_Mfdv_g=0\}\oplus\dg^{-1}(0)\cap tr^{-1}(0)\ee
where $H_F$ denotes the Hessian of $\mathcal{F}$ at $g.$
\end{definition}
Riemannian functionals and their critical points have been topics of interest both in Riemannian geometry and physics. A complete classification of irreducible symmetric spaces of compact type as stable critical points of the total scalar curvature functional is available due to N. Koiso (see \cite{KN1}, \cite{KN2}, \cite{KN3}). Besson, Courtois and Gallois proved that compact quotients of rank one symmetric spaces of non-compact type are global minima for the $L^{\frac{n}{2}} $-norm of scalar curvature (\cite{BCG}). In \cite{GV} Gursky and Viaclovsky studied rigidity, stability and local minimizing properties of Einstein metrics for the quadratic Riemannian functional $\mathcal{F}_t$ defined by
$$\mathcal{F}_t(g)=\mathcal{R}ic(g)+t\mathcal{S}(g), \ \ {\rm for} \ \ t\in \mathbb{R}.$$ 
In this paper, we focus on understanding stability of products of closed Einstein manifolds of the curvature functionals $\mathcal{R},$ $\mathcal{R}ic,$ $\mathcal{F}_t$ and $\mathcal{W}_2$.
  
Let $(M,g)$ be the product of closed Einstein manifolds $(M_0,g_0)$ and $(M_1,g_1)$ with respective Einstein constants $\lambda_0$ and $\lambda_1.$  Then $(M,g)$ is a critical metric of $\mathcal{R}ic$ if and only if $|\lambda_0|=|\lambda_1|.$ This condition can always be achieved after suitable rescaling of $g_1$ or $g_2$, provided both $\lambda_1,\lambda_2\neq 0$. Let $\mathcal{M}_W$ denote the space of doubly warped product metrics on $M$ with constant volume. The warped product variations play an important role in determining the stability of $\mathcal{R}ic$ at products of Einstein metrics. The tangent space of $\mathcal{M}_W$ at $g$ is given by $(C^{\infty}M_0.g_1+C^{\infty}M_1.g_0)$. If $g$ is a critical metric of $\mathcal{R}ic$ and $H_{Ric}$ restricted to $(C^{\infty}M_0.g_1+C^{\infty}M_1.g_0)$ satisfies (\ref{eqn:hessian}), then we say that $g$ is stable for $\mathcal{R}ic$ restricted to $\mathcal{M}_W$.

\begin{thm} \label{main1} Let $(M,g)$ be the product of two closed Einstein manifolds $(M_0,g_0)$ and $(M_1,g_1)$ with respective dimensions $n_0,n_1\geq 3$, Einstein constants $\lambda_0$, $\lambda_1$ and first eigenvalues of the Laplacian $\mu_0$, $\mu_1$. $(M,g)$ is stable for $\mathcal{R}ic$ restricted to $\mathcal{M}_W$ if and only if one of the following conditions is satisfied.
\item
(i) $ \lambda_0=\lambda_1>0$.
\item 
(ii) 
$ \lambda_0=\lambda_1<0$ and for each $j\in \{0,1\}$ either $ n_j =3,4$ or $\frac{\mu_j}{|\lambda_j|}>c(n_{[j+1]})$  where  $c(a)=\frac{a-6+\sqrt{9a^2-36a+4}}{2(a+1)}$ and $[j]=j \mod 2$.
\item
(ii) $\lambda_0=-\lambda_1>0$ and either $ n_0 =3$, or $\frac{\mu_1}{|\lambda_1|} >  \frac{(n_0+2)+\sqrt{9n_0^2-20n_0-28}}{2(n_0+1)}$.
\end{thm}
If $(M_0,g_0)$ and $(M_1,g_1)$ are stable for $\mathcal{R}ic$, then we prove that $H_{Ric}$  restricted to TT-tensors on $M$ satisfies the condition for stability stated in (\ref{eqn:hessian}). Consequently, we have the following theorem.
\begin{thm} \label{main2}  Let $(M,g)$ be the product of two closed Einstein manifolds $(M_0,g_0)$ and $(M_1,g_1)$ with respective dimensions $n_0,n_1\geq 3$. If $(M_0,g_0)$ and $(M_1,g_1)$ are stable for $\mathcal{R}ic$ then $(M,g)$ is stable if and only if $(M,g)$ is stable for $\mathcal{R}ic$ restricted to $\mathcal{M}_W$.
\end{thm}
For instance, the sphere ($S^n$), complex projective space ($\mathbb{C}P^n$) and $S^n\times S^n$ with standard Riemannian metrics are stable for $\mathcal{R}ic$ \cite{GV}. As a consequence of the above theorem $S^n\times \mathbb{C}P^m$ , product of finitly many $S^n$ and  $\mathbb{C}P^m$ possibly with different dimensions are stable for $\mathcal{R}ic$. A compact hyperbolic manifold $H^n$ is stable for $\mathcal{R}ic$ if, the first eigenvalue of the laplacian $\mu>\frac{2(n-4)(n-1)}{n}$ (see Remark 7.4, \cite{GV}). 
\begin{thm}\label{SnHn} Let $(M,g)$ be the product of $S^n$ and a compact hyperbolic manifold $H^m$, then $(M,g)$ is stable for $\mathcal{R}ic$ if and only if either $m=3$ or the first eigenvalue of the Laplacian $\mu$ of $H^m$ satisfies
$$\frac{\mu}{m-1}>\max \left\{\frac{2(m-4)}{m}, \frac{(n+2)+\sqrt{9n^2-20n-28}}{2(n+1)}\right \}$$
\end{thm}
 
In particular, if $\mu>2(m-1)$, then $S^n\times H^m$ is always stable for $\mathcal{R}ic.$ A compact hyperbolic manifold with dimension $4$ is always stable for $\mathcal{R}ic$ (Remark 7.4, \cite{GV}). The above theorem implies that  
if $\mu< 6$ and $n$ sufficiently large then $S^n\times H^4$ is not stable for $\mathcal{R}ic.$ More generally, if $M=S^n\times H^m$ and
$$\frac{2(m-4)}{m}<\frac{\mu}{m-1}<\frac{(n+2)+\sqrt{9n^2-20n-28}}{2(n+1)}$$
then $S^n$ and $H^m$ both are individually stable for $\mathcal{R}ic$ but the product $S^n\times H^m$ is not stable for $\mathcal{R}ic.$ Similar results also hold for products of $\mathbb{C}P^n$ and compact quotients of its non-compact dual.

We also prove similar stability criteria for critical metrics of $\mathcal{F}_t$ those are products of Einstein metrics with Einstein constants of opposite signs (see Theorem \ref{st-F_t-opposite} for details) thereby generalizing the work of Gursky and Viaclovsky \cite{GV}. Due to Chern-Gauss-Bonnet theorem $\mathcal{W}_2$ is equivalent to $\mathcal{F}_{-\frac{1}{3}}$ in dimension $4$. Hence rigidity or stability of $\mathcal{F}_{-\frac{1}{3}}$ implies rigidity or stability of Bach flat metrics in this case \cite{GV}. 
Since $\mathcal{W}_2$ is not conformally invariant for $n>4$ we consider the conformally invariant Riemannian functional $\mathcal{W}_{\frac{n}{2}}$ defined by  
$$\mathcal{W}_{\frac{n}{2}}(g)= \int_M |W_g|^{\frac{n}{2}} dv_g. $$
Stability of $S^p\times S^q$ and $\mathbb{C}P^n$ for $\mathcal{W}_{\frac{n}{2}}$ has been studied in \cite{KO},  \cite{GLW}, \cite{GL}. We prove the following result. 
\begin{thm} \label{main3}
Let $(M^n,g)$ be a closed manifold with dimension $n\geq 6$. If the Riemannian universal cover of $(M,g)$ is the product of a sphere $S^k$ and a hyperbolic space $\mathbb{H}^{n-k}$ ($k,n-k\geq 3$) with the standard product metric, then $(M^n,g)$ is stable for $\mathcal{W}_{\frac{n}{2}}$.
\end{thm}
Locally conformally flat manifolds are global minima for $\mathcal{W}_{\frac{n}{2}}$. Topology and rigidity of these metrics with non-negative curvature conditions have been studied in \cite{SY}, \cite{Zhu}. Whereas some examples of locally conformally flat metrics with negative ricci curvature are constructed in \cite{CH} using warped product metrics with constant curvature manifolds as base manifolds. Theorem \ref{main3} implies that $S^n\times H^m$ with standard product metric $g$ is a strict local minimizer for $\mathcal{W}_{\frac{n}{2}}$. Hence there exists a neighbourhood $\mathcal{U}$ of $g$ in $C^{k,\alpha}$-topology such that if $\tilde{g}\in \mathcal{U}$ is a locally conformally flat metric then $c\tilde{g}, (c>0)$ is isometric to a metric conformal to $g$. Here $k$ is sufficiently large. $S^n\times H^n$ is also stable for $\mathcal{R}_2$ when the first eigenvalue of the laplacian of $H^n$ is sufficiently large (see Section 5.3).
\section{Preliminaries}
In this section, we setup the notations and recall a few definitions to be followed throughout this paper. Let $(M^n,g)$ be a closed Riemannian manifold of dimension $n$. $\langle , \rangle$, $\langle , \rangle_{L^2}$  respectively denote the point-wise and global inner products induced by $g$ on tensors and $|.|$, $\|.\|$ be the corresponding norms. Consider an orthonormal frame $\{e_i\}$ around a point on $M$. 
 
$\Rt$ is a symmetric $2$-tensor defined by
$$\Rt(x,y)=\sum_{i,j,k} R(x,e_i,e_j,e_k)R(y,e_i,e_j,e_k).$$

The self-adjoint operator $\Rc: S^2M \to S^2M$ is given by
$$\Rc(h)(x,y):=\sum_{i,j} R(e_i,x,e_j,y)h(e_i,e_j).$$

The Kulkarni-Nomizu product of $h_1$, $h_2\in S^2M$ is defined as
$$h_1\wedge h_2(x,y,z,w)=\frac{1}{2}[h_1(x,z)h_2(y,w)+h_1(y,w)h_2(x,z)-h_1(x,w)h_2(y,z)-h_1(y,z)h_2(x,w)].$$
Let $D$ be the Riemannian connection on $(M,g)$ and $D^*$ its formal adjoint. The divergence operator $\dg: S^2(M) \to \Omega^1(M)$ is defined by
$\dg(h)(x)=-\sum_i D_{e_i}h(e_i,x).$ For $h\in S^2M$ define
$$ d^Dh(x,y,z)=D_xh(y,z)-D_yh(x,z).$$
$\dg^*$ and $\delta^D$ denote the adjoints of $\dg^*$ and $d^D$ respectively. Then from \cite{Ber} we have,
$${\delta}^Dd^Dh = 2D^*D h+r\circ h+h\circ r-2\Rc(h)-\dg^*\dg h.$$

$h\in S^2(M)$ is said to be a transverse traceless tensor (or a TT-tensor) if $\dg (h)=0$ and $tr(h)=0$. \\
The Lichnerowicz Laplacian defined on TT-tensors is given by
\Bea 
\la_L h =D^*D h+r\circ h+h\circ r-2\Rc(h).
\Eea
Let $\la$ denote the Hodge-Laplcian acting on differential forms. Then the  B\"ochner-Weitzenb\"ock formula for differential $1$-forms is given by
$$\la\alpha=D^*D\alpha+r(\alpha) \ \ \forall \alpha \in \Omega^1(M) \ \ {\rm where} \  \ \la f=-tr Ddf \ \ \forall f\in C^{\infty}M.$$
If $\delta_g\alpha=0$ and $(M,g)$ is Einstein, then B\"ochner-Weitzenb\"ock formula implies that,
\bea \label{eqn-BW}
\|d\alpha\|^2 = \langle \la \alpha , \alpha \rangle_{L^2} = \langle D^*D \alpha , \alpha \rangle_{L^2} + \lambda \|\alpha \|^2 
=\|D\alpha\|^2+\lambda \|\alpha\|^2 \eea
We conclude this section by stating some first variational formulae from \cite{BE} Theorem 1.174. Given a one-parameter family $g(t)$ of Riemannian metrics with $g'(0)=\frac{d}{dt}(g(t))|_{t=0} = h$, let $T'_g(h)=\frac{d}{dt}(T(t))|_{t=0}$, where $T(t)=T(g(t))$ is a tensor depending on $g(t)$. Then $D'_g(h), R'_g(h),r'_g(h)$ and $s_g'(h)$ are given by 
\bea\label{variation1}
\nonumber &(i)& C_h(x,y,z)= g(D'_g(h)(x,y),z)=\frac{1}{2}(D_xh(y,z)+D_yh(x,z)-D_zh(x,y))\\
\nonumber &(ii)& R'_g(h)(x,y,z,w)=\frac{1}{2}[D^2_{y,z}h(x,w)+D^2_{x,w}h(y,z)-D^2_{x,z}h(y,w)-D^2_{y,w}h(x,z)\\
\nonumber && \; \; \; \; \; \; \; \; \; \;\; \; \; \; \;\; \; \; \; \;\; \; \; \; \;\; \; \; \; \;\; \; \; \; \;\; \; \; \; \; \;\;\;\;\;\;  \; \; \; \; \;\;\;\;\;\; \; \; \;   
+ h(R(x,y,z),w) - h(R(x,y,w),z)]\\
\nonumber &(iii)& r'_g(h)= \frac{1}{2}(D^*D h+r\circ h+h\circ r-2\Rc(h)-2\dg^*\dg h-Dd(trh))\\
&(iv)& s'_g(h)=\la trh +\dg^2(h)-(r,h)
\eea
\section{Some variational formulae at products of Einstein manifolds}
In this section we compute first variational formulae for tensors associated to gradients of $\mathcal{R},\mathcal{R}ic$, $\mathcal{S}$ and $\mathcal{W}_2$ at products of Einstein manifolds. Gradients of these functionals restricted to the space of Riemannian metrics with constant volume are given below \cite{BE}.
\Bea
&&\nabla\mathcal{R}(g)=2{\delta}^Dd^Dr-2\Rt+\frac{1}{2}|R|^2g+(\frac{2}{n}-\frac{1}{2})\|R\|^2g,\\
&&\nabla\mathcal{R}ic(g)={\delta}^Dd^Dr -D^{*}Dr-2r\circ r+\frac{1}{2}(\la s)g+\frac{1}{2}|r|^2g+(\frac{2}{n}-\frac{1}{2})\|r\|^2g,\\
&&\nabla\mathcal{S}(g)=2Dds +2(\la s)g-2s.r+\frac{1}{2}|s|^2g+(\frac{2}{n}-\frac{1}{2})\|s\|^2g.
\Eea
From the decomposition of Riemannian curvature tensor $R= \frac{s}{n(n-1)}g\wedge g +\frac{2}{n-2}(r-\frac{s}{n}g)\wedge g +W$
\bea\label{Weyl1}
|W_g|^2= |R_g|^2- \frac{4}{(n-2)} \left(|r|^2 -\frac{s^2}{2(n-1)}\right).
\eea 
Consequently,
\bea\label{weyl2}
\nabla \mathcal{W}_2(g)=\nabla\mathcal{R}(g)- \frac{4}{n-2}\left(\nabla\mathcal{R}ic(g)- \frac{1}{2(n-1)}\nabla\mathcal{S}(g)\right).
\eea
Let $(M^{n_0+n_1},g)$ ($n_0+n_1 \geq 4$) be the product of closed  Einstein manifolds $(M_0^{n_0},g_0)$ and $(M_1^{n_1},g_1)$ with respective Einstein constants $\lambda_0$ and $\lambda_1$. Consider divergence free one forms $\alpha_i$ on $M_i$ such that $\|\alpha\|=1$ for each $i=0,1.$ Then $h=\alpha_0 \odot \alpha_1= \alpha_0\otimes \alpha_1 + \alpha_1 \otimes \alpha_0 \in S^2(M)$ survives only on mixed two planes. As 
$$\delta_g(h)= \delta_{g_0}(\alpha_0)\alpha_1 +\delta_{g_1}(\alpha_1)\alpha_0$$
$h$ is also divergence free. Hence $h$ is a TT-tensor on $M$. 
Let $f \in C^{\infty}M_{i+1}$ with $\int_{M_{i+1}}f=0$ for $i=0,1$ and indices are taken modulo 2. Then $\dg (fg_{i})=0$, $Tr_g(fg_i)=\int_{M}tr_g(fg_i)=0$ and the following identities hold
\bea \label{eqn-Df}
D(fg_i)=df\otimes g_i \ ,\  D^2(fg_i)=Ddf\otimes g_i \ ,\  D^*D (fg_i)=(\la f)g_i.
\eea
$\alpha_0\odot \alpha_1$ and $fg_i$ as described above satisfy following Lemmas. We use Einstein summation convention.  
\begin{lem}\label{lem1}
\Bea
&&(i) \ \langle ({\delta}^Dd^Dr)'(fg_i), fg_i \rangle_{L^2} =n_i(\|\la f\|^2-(\lambda_0+\lambda_1)\|df\|^2)\\
&&(ii)\ \langle ({\delta}^Dd^Dr)'(\alpha_0 \odot \alpha_1 ), \alpha_0\odot\alpha_1 \rangle_{L^2}=2(\| \la \alpha_0 \|^2 + \|\la \alpha_1 \|^2 +2 \langle \la \alpha_0,\alpha_0 \rangle \langle \la \alpha_1,\alpha_1 \rangle)\\
&& \hspace{2.4in}-\frac{5}{2}(\lambda_0+\lambda_1)(\langle \la \alpha_0,\alpha_0 \rangle+\langle \la \alpha_1,\alpha_1 \rangle)+4\lambda_0\lambda_1
\Eea
\end{lem}
\begin{proof}
Since $r$ is parallel $(({\delta}^D)'h)d^Dr=0$. Hence,
\bea \label{eqn-dD}
\langle ({\delta}^D{d^D}r)'(h), h \rangle_{L^2} = \langle {\delta}^D(d^D)'(h)r+ {\delta}^D{d}^D(r'(h)),h \rangle_{L^2} = \langle (d^D)'(h)r,d^Dh\rangle_{L^2} + \langle r'(h),{\delta}^Dd^Dh\rangle_{L^2}
\eea
$(d^D)'(h)r$ can be expressed in terms of an orthonormal basis $\{e_i\}$ as
\begin{eqnarray*}
((d^D)'(h)r)_{pql}&=& r_{ip}C_h(e_q,e_l, e_i)- r_{iq} C_h(e_p,e_l, e_i)
\end{eqnarray*}
Therefore,
\begin{eqnarray*}
\langle (d^D)'(h)r,d^Dh \rangle &=&[r_{pi}C_{jkp}-r_{jp}C_{ikp}][Dh_{ijk}-Dh_{jik}]\\
&=& 2[r_{pi}C_{jkp}-r_{jp}C_{ikp}]Dh_{ijk}\\
&=& Dh_{ijk}\{r_{pi}(Dh_{jkp}+Dh_{kjp}-Dh_{pjk})-r_{jp}(Dh_{ikp}+Dh_{kip}-Dh_{pik})\}\\
&=&  [-r_{ii}(Dh_{ijk})^2-r_{jj}(Dh_{ijk})^2+3r_{ii}Dh_{ijk}Dh_{jik}-r_{jj}Dh_{ijk}Dh_{kij}]\\
 &=& (3r_{ii}-r_{kk})Dh_{jik}Dh_{ijk}-(r_{ii}+r_{jj})(Dh_{ijk})^2
 \Eea
Putting $h=fg_1$ and using formulae in (\ref{eqn-Df}) we have,
\Bea( (d^D)'(fg_1)r,d^D(fg_1) ) &=& (3r_{ii}-r_{kk})df_jg_{1ik}df_ig_{1jk}-(r_{ii}+r_{kk})[df_ig_{1kj}]^2\\
                                &=&= -n_1(\lambda_0+\lambda_1)|df|^2
\Eea
From (\ref{variation1}) and (\ref{eqn-Df}) one obtains \
$\langle r'(fg_1),{\delta}^Dd^D(fg_1)\rangle_{L^2}= n_1\|\la f\|^2.$  \ Therefore, 
$$\langle (\delta^Dd^Dr)'(fg_1),fg_1) \rangle_{L^2} = n_1\|\la f\|^2-n_1(\lambda_0+\lambda_0)\|df\|^2.$$
Considering $h=\alpha_0\odot \alpha_1$ in (\ref{eqn-dD}), 
\Bea
\langle r'(\alpha_0\odot \alpha_1),{\delta}^Dd^D(\alpha_0\odot \alpha_1)\rangle_{L^2} &=& \|D^{*}D (\alpha_0\odot \alpha_1)\|^2+ \frac{1}{2}(\lambda_0+\lambda_1)^2\|\alpha_0\odot \alpha_1\|^2\\
&&\ \  + \frac{3}{2}(\lambda_0+\lambda_1)\langle D^{*}D(\alpha_0\odot \alpha_1),\alpha_0\odot \alpha_1\rangle_{L^2}
\Eea
\Bea \langle (d^D)'(\alpha_0\odot\alpha_1)r,d^D(\alpha_0\odot \alpha_1) \rangle 
 &=& (3r_{ii}-r_{kk})Dh_{jik}Dh_{ijk}-(r_{ii}+r_{jj})(Dh_{ijk})^2\\
 &=& (3\lambda_0-\lambda_1)|\alpha_1 |^2 D(\alpha_0)_{ij}D(\alpha_0)_{ji}+(3\lambda_1-\lambda_0)|\alpha_0 |^2 D(\alpha_1)_{ij}D(\alpha_1)_{ji}\\
 &&\;\; -(3\lambda_0+\lambda_1)|\alpha_1|^2 |D\alpha_0|^2- (3\lambda_1+\lambda_0)|\alpha_0 |^2 |D\alpha_1|^2.
\Eea
Since $D(\alpha_i)_{pq}D(\alpha_i)_{qp} = |D\alpha_i|^2-\frac{1}{2}|d\alpha_i|^2$ one obtains,
\Bea \langle (d^D)'(\alpha_0\odot\alpha_1)r,d^D(\alpha_0\odot \alpha_1) \rangle &=&-\frac{1}{2}(3\lambda_0-\lambda_1)|\alpha_1|^2|d\alpha_0|^2-2\lambda_1 |\alpha_1|^2|D\alpha_0|^2\\
&& -\frac{1}{2}(3\lambda_1-\lambda_0)|\alpha_0|^2|d\alpha_1|^2-2\lambda_0 |\alpha_0|^2|D\alpha_1|^2 
\Eea
Using (\ref{eqn-BW}) we have,    
\begin{eqnarray*}
\langle (d^D)'(h)r,d^Dh \rangle_{L^2} = -\frac{1}{2}(3\lambda_0^2+3\lambda_1-2\lambda_0\lambda_1)\|\alpha_0\|^2\|\alpha_1\|^2- \frac{3}{2}(\lambda_0+\lambda_1)(\|D\alpha_0 \|^2+\|D\alpha_1 \|^2)
\end{eqnarray*}
Putting $\|\alpha_i\|=1$ and combining all the expressions above, the result follows.
\end{proof}
%%%%%%%%%%%%%%%%%%%%%%%%%%%%%%%%%%%%
\begin{lem}\label{lem2}
\Bea
&& (i) \ \langle (D^{*}Dr)'(fg_i), fg_i \rangle_{L^2}=\frac{n_i}{2}\left(\|\la f\|^2-2\lambda_i\|df\|^2\right)\\ 
&& (ii) \ \langle (D^{*}Dr)'(\alpha_1\odot \alpha_2), \alpha_1\odot \alpha_2 \rangle_{L^2} =\| \la \alpha_1 \|^2 + \|\la \alpha_2 \|^2 +2 \langle \la \alpha_1,\alpha_1 \rangle_{L^2} \langle \la \alpha_2,\alpha_2 \rangle_{L^2}+ 4 \lambda_1\lambda_2\\
&& \hspace{2.3in}-\frac{1}{2}(3\lambda_1+5\lambda_2)\langle \la \alpha_1,\alpha_1 \rangle_{L^2}-\frac{1}{2}(3\lambda_2+5\lambda_1)\langle \la \alpha_2,\alpha_2 \rangle_{L^2} 
\Eea
\end{lem}
\begin{proof}
Since the Ricci tensor is parallel $D^{*}(D'(h)r)=0.$ Therefore,
\Bea
\langle (D^{*}Dr)'(h),h\rangle_{L^2}= \langle D^{*}(D'(h)r)+ D^{*}D(r'(h)),h \rangle_{L^2} = \langle (D)'(h)r,Dh\rangle_{L^2} + \langle r'(h),D^{*}Dh\rangle_{L^2}
\Eea
With respect to an orthonormal basis $D_g'(h)r$ can be written as, 
$$(D'(h)r)_{pql}=- r_{il}C_{pqi}- r_{iq}C_{pli}=-r_{ll}C_{pql}-r_{qq}C_{plq}=-2r_{qq}C_{plq}$$
For $h=fg_1$ one has $\langle r'(h),D^{*}Dh\rangle_{L^2}= \frac{1}{2}\|\la f\|^2$.
\Bea
\langle (D'(fg_1)r),D(fg_1) \rangle = -2r_{qq}C_{plq}df_pg_{1ql}= -r_{qq}(df_p)^2g_{1qq} = -n_1\lambda_1|df|^2.
\Eea
Therefore,
$$\langle (D^{*}Dr)'(fg_1),fg_1\rangle_{L^2} =\frac{n_1}{2}\|\la f\|^2-n_1\lambda_1\|df\|^2.$$
Considering $h=\alpha_1\odot \alpha_2$ we have,
\Bea
&&\langle r'(\alpha_1\odot\alpha_2),D^*D(\alpha_1\odot \alpha_2)\rangle_{L^2}=\frac{1}{2}\|D^*D(\alpha_1\odot\alpha_2)\|^2+\frac{1}{2}\|D(\alpha_1\odot \alpha_2)\|^2
\Eea
Also in this case, $\langle (D'(h)r), Dh \rangle =- r_{kk}(Dh_{ijk}+ Dh_{jik}-Dh_{kji})Dh_{ijk}.$ Hence,
\Bea
\langle (D'(\alpha_0\odot\alpha_1)r), D(\alpha_0\odot\alpha_1) \rangle
&=& -(\lambda_0-\lambda_1)(|\alpha_1|^2 (D\alpha_{0ij}D\alpha_{0ji})-|\alpha_0 |^2 (D\alpha_{1ij}D\alpha_{1ji}))\\
  &&  -\frac{1}{2}(\lambda_0+\lambda_1)|D(\alpha_0\odot\alpha_1)|^2\\
&=& -\frac{3\lambda_0+\lambda_1}{2}|D\alpha_0|^2|\alpha_1|^2-\frac{3\lambda_2+\lambda_1}{2}|D\alpha_1|^2|\alpha_0|^2\\
&& + \frac{1}{2}(\lambda_0-\lambda_1)^2|\alpha_0|^2|\alpha_1|^2
\Eea
\noindent Putting $\|\alpha_i\|=1$ ($i=1,2$) and using (\ref{eqn-BW}) we obtain the desired result.
\end{proof}
\begin{lem}\label{lem3}
\Bea
&& (i) \ \langle(r\circ r)'(fg_i),fg_i\rangle_{L^2}=\lambda_i n_i(\|df\|^2-\lambda_i\|f\|^2)\\
&& (ii) \ \langle(r\circ r)'(\alpha_1\odot \alpha_2),\alpha_1\odot \alpha_2\rangle_{L^2} = (\lambda_1+\lambda_2)(\langle \la \alpha_1,\alpha_1 \rangle_{L^2}+\langle \la \alpha_2,\alpha_2 \rangle_{L^2})-2\lambda_1\lambda_2
\Eea
\end{lem}
\begin{proof}
In local coordinates $r\circ r$ can be described as $(r\circ r)_{pq}=g^{ij}r_{pi}r_{qj}.$ Therefore,
\begin{eqnarray*}
[(r\circ r)'(h)]_{pq}&=& (g^{ij})'r_{pi}r_{qj}+ g^{ij}r'_{pi}r_{qj}+g^{ij}r_{pi}r'_{qj}\\
&=& -g^{ik}h_{kl}g^{lj}r_{pi}r_{qj}+ (r'(h)\circ r)_{pq}+(r\circ r'(h))_{pq}\\
&=& -(r\circ h\circ r)_{pq}+ (r'(h)\circ r+r\circ r'(h))_{pq}.
\end{eqnarray*}
Hence, \ \ $\langle (r\circ r)'(h),h\rangle = - r_{pp}r_{qq} (h_{pq})^2 +\langle r'(h) , h\circ r+r\circ h\rangle.$\\
Putting $h=fg_1$ in the above equation we have,
\Bea 
\langle (r\circ r)'(fg_1),fg_1\rangle_{L^2} &=& -\lambda_1^2n_1\|f\|^2+n_1\lambda_1\|df\|^2
\Eea
When $h=\alpha_0\odot \alpha_1$, one has
\Bea
\langle (r\circ r)'(\alpha_0\odot\alpha_1),\alpha_0\odot\alpha_1\rangle_{L^2}
 &=& \frac{1}{2}(\lambda_0^2+\lambda_1^2)\|\alpha_0\odot\alpha_1\|^2+\frac{1}{2}\|D(\alpha_0\odot\alpha_1)\|^2
\Eea
The Lemma now follows from (\ref{eqn-BW}).
\end{proof}
\begin{lem}\label{lem4}
Let $(M^{n_0+n_1},g)$ be the product of a spherical space form  $(M_0,g_0)$ and a compact hyperbolic manifold $(M_1,g_1)$ with constant sectional curvatures $1$ and $-1$ respectively. Then
\Bea && (i) \ \langle (\Rt)'(fg_i),fg_i\rangle_{L^2}=-\|R^i\|^2\|f\|^2;\\
&& (ii) \ (|R|^2)'(fg_i)=-2f|R^i|^2;\\
&& (iii) \ \langle (\Rt)'(\alpha_0 \odot \alpha_1),\alpha_0 \odot \alpha_1\rangle_{L^2} =2\langle \la \alpha_0,\alpha_0\rangle_{L^2} -2\langle \la \alpha_1,\alpha_1 \rangle_{L^2}-4(n_0+n_1-2)
\Eea
where $R^i$ denotes the Riemannian curvature tensor of $(M_i,g_i)$.
\end{lem} 
\begin{proof} With respect to an arbitrary basis $\Rt$ can be written as
\Bea \Rt_{pq}=g^{i_1i_2}g^{j_1j_2}g^{k_1k_2}R_{pi_1j_1k_1}R_{qi_2j_2k_2}.
\Eea
Differentiating each terms and putting
$(g^{ij})'=-g^{im}h_{mn}g^{nj}$ we have,
\Bea (\Rt_g)'(h)_{pq}&=&-h_{mn}\left(R_{pmij}R_{qnij}+R_{pimj}R_{qinj}+R_{pijm}R_{qijn}\right) +(R'_g(h))_{pijk}R_{qijk}+R_{pijk}(R'_g(h))_{qijk}\\
&=&-h_{mn}[R_{pmij}R_{qnij}+2 R_{pimj}R_{qinj}]+(R'_g(h))_{pijk}R_{qijk}+R_{pijk}(R'_g(h))_{qijk}
\Eea
Setting $$L(x,y,z,w)=L_h(x,y,z,w)=D^2_{y,z}h(x,w)+D^2_{x,w}h(y,z)-D^2_{x,z}h(y,w)-D^2_{y,w}h(x,z)$$
and using symmetries of $R$ and $h$, one obtains 
\Bea ((R'_g(h))_{pijk}R_{qijk}+R_{pijk}(R'_g(h))_{qijk})h_{pq}&=&2 h_{pq}(R'_g(h))_{pijk}R_{qijk}=  \{L_{pijk}h_{pq}R_{qijk}+2 R_{pimj}R_{qinj}h_{mn}\}
\Eea
Therefore,
\Bea \langle (\Rt)'(h),h\rangle&=&(\Rt_g)'(h)_{pq}h_{pq}\\
&=&\{L_{pijk}R_{qijk}h_{pq}-h_{mn}h_{pq}R_{pmij}R_{qnij}\}
\Eea
Let $h=fg_i$. Using (\ref{eqn-Df} ) we have $L=Ddf\wedge g_i$ and
\Bea 
\langle (\Rt)'(fg_i),fg_i\rangle=\{fL_{pqjk}R^1_{pqjk}-f^2R^i_{pmqj}R^i_{pmqj}\}= \langle Ddf\wedge g_i, fR^i\rangle-|f|^2|R^i|^2= -|f|^2|R^i|^2.
\Eea
Also, since $|R|^2=tr(\Rt)$, one has
\Bea 
(|R|^2)'(fg_i)&=&\langle(\Rt)'(fg_i),g\rangle-\langle \Rt,fg_i\rangle= -2f|R^i|^2
\Eea
If $h=\alpha_0\odot \alpha_1$ then 
\Bea
h_{mn}h_{pq}R_{pmij}R_{qnij}=R(\alpha_0,\alpha_1,e_i,e_j)R(\alpha_0,\alpha_1,e_i,e_j)=0
\Eea
Also,
\Bea  
h_{pq}L_{pijk}R_{qijk}&=&  h_{pq}[D^2h_{pkij}+D^2h_{ijpk}-D^2h_{pjik}-D^2h_{ikpj}]R_{qijk}= 2h_{pq}[D^2h_{pkij}-D^2h_{ikpj}]R_{qijk}
\Eea
Since $R$ is parallel, $h_{pq}D^2h_{pkij}R_{qijk}=-\langle \dg^*\dg\Rc(h)),h\rangle=0.$
Finally, one obtains
\Bea 
L_{pijk}R_{qijk}h_{pq}
&=& 2 D^2h_{ijpk}R_{qijk}h_{pq} \\ 
&=& 2[((\alpha_0)_p)^2 (\alpha_1)_q (D^2 \alpha_1)_{ijk}R_{qijk}+ ((\alpha_1)_p)^2 (\alpha_0)_q (D^2 \alpha_0)_{ijk}R_{qijk}]\\
&=&2|\alpha_1|^2(D^2(\alpha_0)_{iji}-D^2(\alpha_0)_{iij})(\alpha_0)_j-2|\alpha_0|^2(D^2(\alpha_1)_{iji}-D^2(\alpha_1)_{iij})(\alpha_1)_j\\
&=& 2(|\alpha_1|^2|D\alpha_0|^2-|\alpha_1|^2|D\alpha_0|^2)+ 2|\alpha_1|^2 D^2(\alpha_0)_{iji}(\alpha_0)_j -2|\alpha_0|^2D^2(\alpha_1)_{iji}(\alpha_1)_j
\Eea
The Ricci-identity for $1$-forms implies, 
\Bea 
D^2(\alpha_k)_{iji} (\alpha_k)_{j} =((D^2(\alpha_k)_{jii}-\alpha_k(R_{jii})) (\alpha_k)_{j}= -e_j(\delta\alpha_k)(\alpha_k)_j- (r_k)_{jj}(\alpha_k)_j^2
\Eea 
Using $\delta(\alpha_k)=0$ and putting $r_0=n_0-1, r_1=-(n_1-1)$ one obtains
\Bea 
\langle (\Rt)'(h),h\rangle_{L^2} = 2(\|\alpha_1\|^2\|D\alpha_0\|^2-\|\alpha_0\|^2\|D\alpha_1\|^2)-2(n_0+n_1-2)\|\alpha_0\|^2\|\alpha_1\|^2
\Eea
Finally, putting $\|\alpha_0\|=1=\|\alpha_1\|$, and using (\ref{eqn-BW}) the result follows.
\end{proof}
\section{Second variations of $\mathcal{R}ic$}
Let $(M,g)=(M_0\times M_1, g_0+g_1)$ be as described in Section 3 with
$|\lambda_0|=|\lambda_1 |.$ A traceless symmetric $2$-tensor $h$ on $M$ decomposes as \cite{KO}
$$ h=h_1+\hat{h}+h_2+f(n_1g_0-n_0g_1)$$
where $h_1,h_2$ are TT-tensors tangent to $M_1$ and $M_2$ respectively, $\hat{h}$ is a TT-tensor on $M$ surviving only on mixed planes, and $f\in C^{\infty}M$. This decomposition is orthogonal. Let $S^2_1M, S^2_0M, S^2_2M$ denote the first three components of the above decomposition respectively. Then 
\be \label{eqn:decomposition2}
S^2M=Im\dg^*\oplus S^2_1M\oplus S^2_0M\oplus S^2_2M\oplus (C^{\infty}Mg_1+C^{\infty}Mg_2)
\ee
\subsection{Transverse-traceless variations :} From the above decomposition it is easy to see that if $h$ is a TT-tensor on $M$ then $h\in  S^2_1M\oplus S^2_0M\oplus S^2_2M\oplus \mathbb{R}.(n_1g_0-n_0g_1).$ The last component is the span of parallel tensor $n_1g_0-n_0g_1.$ 
\begin{thm}\label{Ric:TT} 
Let $(M^{n_0+n_1},g)$ ($n_0,n_1\geq 3$) be the product of two closed Einstein manifolds $(M^{n_0}_0,g_0)$ and $(M^{n_1}_1,g_1)$ with respective Einstein constants $\lambda_0,$ $\lambda_1$. If $(M,g)$ is a critical metric of $\mathcal{R}ic$ then there exists a constant $C>0$ such that for any TT-tensor $h$ on $M$,
\be \label{eqn:ric-TT}
\langle (\nabla\mathcal{R}ic)'(h), h \rangle_{L^2} \geq C\|h\|^2 
\ee
if and only if the Hessian of $\mathcal{R}ic$ restricted to the TT-tensors on $(M_i,g_i)$ satisfies the same condition for each $i=0,1$.
\end{thm}
\begin{proof} 
Let $H$ denote the Hessian of $\mathcal{R}ic$ at $g$. Using Lemmas \ref{lem1},\ref{lem2} and \ref{lem3}, it is easy to see that the decomposition given by  (\ref{eqn:decomposition}) is orthogonal with respect to $H$. Therefore, it is sufficient to prove (\ref{eqn:ric-TT}) component-wise.

The standard projection map from $M$ to $M_i$ induce an isomorphism between $S^2_iM$ to the space of TT-tensors on $M_i$ for each $i.$  For any $h\in S^2_1M$ there exists a two parameter family of metrics $g_{s,t}$ on $M_0$ with  $g_{0,0}=g_0$ such that 
$$H(h,h)=\frac{\partial}{\partial s}\frac{\partial}{\partial t} \mathcal{R}ic(g_{s,t}+g_2)_{|(0,0)}.$$
Using Riemannian product structure we see that $H$ restricted to $S^2_1M$ has a positive lower bound if and only if the Hessian of $\mathcal{R}ic$ on $(M_1,g_1)$ acting on TT -tensors satisfies (\ref{eqn:ric-TT}). Similar argument holds for $S^2_2M$. 

Let $\alpha_0$ and $\alpha_1$ be two divergence free $1$-forms on $M_0$ and $M_1$ with $\|\alpha_0\|=\|\alpha_1\|=1$. Define, $h=\alpha_0\odot \alpha_1$. The set of all two tensors of this type is dense in $S^2_0M$. Therefore to prove that $H$ satisfies (\ref{eqn:ric-TT}) on $S^2_0M$, it is sufficient to prove that $H$ satisfies the same for tensors of type $\alpha_0\odot \alpha_1.$  The following possibilities arise. 
\\
\\
\textit{Case} $(\lambda_0=-\lambda_1=\lambda>0)$: If $\mu_0$ denotes the first eigenvalue of the laplacian on $\Omega^1(M_0)$. (\ref{eqn-BW}) implies that $\mu_0 >\lambda.$ Now using Lemmas \ref{lem1},~\ref{lem2} and \ref{lem3} we have,
\Bea \langle (\nabla\mathcal{R}ic)'(h), h \rangle_{L^2} &=& \langle (\delta^Dd^Dr)'(h)- (D^*Dr)'(h)-2(r\circ r)'(h), h \rangle_{L^2}+ \frac{2}{n_0+n_1} \|r \|^2 \|h \|^2\\
&=&\| \la \alpha_0 \|^2 + \|\la \alpha_1 \|^2 +2 \langle \la \alpha_0, \alpha_0  \rangle_{L^2} \langle \la \alpha_1, \alpha_1 \rangle_{L^2} -\lambda\langle \la \alpha_0, \alpha_0 \rangle_{L^2} +\lambda \langle \la \alpha_1, \alpha_1 \rangle_{L^2}
\Eea
Since,
\Bea \| \la \alpha_0 \|^2- \lambda \langle \la \alpha_0, \alpha_0  \rangle_{L^2} 
&=& \|\alpha_0-\lambda \alpha_0\|^2+\lambda \langle \la \alpha_0-\lambda \alpha_0,\alpha_0\rangle_{L^2} \geq \mu(\mu-\lambda)>0 
\Eea
and the remaining terms in the expression of $\langle (\nabla\mathcal{R}ic)'(h), h \rangle$ are non-negative, the proof follows.
\\
\\
\textit{Case} $(\lambda_0=\lambda_1=\lambda)$ : Let $\lambda$ be the Einstein constant of both $(M_i,g_i)$. From Lemmas \ref{lem1},~\ref{lem2} and \ref{lem3} we have,
\Bea 
\langle (\nabla\mathcal{R}ic)'(h), h \rangle_{L^2} &=&
\| \la \alpha_0 \|^2 + \|\la \alpha_1 \|^2 +2 \langle \la \alpha_0, \alpha_0  \rangle_{L^2} \langle \la \alpha_1, \alpha_1 \rangle_{L^2}+ 8\lambda^2-5\lambda \left(\langle \la \alpha_0, \alpha_0 \rangle_{L^2} + \langle \la \alpha_1, \alpha_1 \rangle_{L^2}\right)
\Eea
Clearly $\lambda<0$, then the proof follows. When $\lambda>0$ then using (\ref{eqn-BW}) we have, 
\Bea \langle (\nabla\mathcal{R}ic)'(h), h \rangle_{L^2} &=& \|D^*D\alpha_0\|^2+\|D^*D\alpha_1\|^2-\lambda\|D\alpha_0\|^2-\lambda\|D\alpha_1\|^2+12\lambda^2\\
&=&\|D^*D\alpha_0-\frac{\lambda}{2}\alpha_0\|^2 +\|D^*D\alpha_1-\frac{\lambda}{2}\alpha_1\|^2+(12-\frac{1}{4})\lambda^2\\
&&> 11 \lambda^2 
\Eea
\end{proof}
%-----------------------------------------
\subsection{Conformal Variations:} 
In this section we study $H$ restricted to $C^{\infty}M_1.g_0+C^{\infty}M_0.g_1.$ Let $\mathcal{M}(g)$ denote the space of all unit volume metrics conformal to $g.$ The tangent space of $\mathcal{M}(g)$ consists of $2$-tensors of the form $fg$ with $\int_M fdv_g=0.$
\begin{lem}\label{lem6} 
Let $(M,g)$ be the product of two closed Einstein manifolds $(M_0^{n_0},g_0)$ and $(M_1^{n_1},g_0)$ with respective Einstein constants $\lambda_0$ and $\lambda_1$ such that $|\lambda_i|=\lambda.$ Let $h=fg_i$ where $f\in C^{\infty}M_{i+1}$. Then
\Bea
\langle (\nabla\mathcal{R}ic)'(h), h\rangle_{L^2} =\begin{cases}
\frac{n_i(n_i+1)}{2}\|\la f\|^2-\frac{\lambda_i n_i(n_i+2)}{2}\|df\|^2- \lambda^2 n_i(n_i-4) \|f\|^2, \text{ if } \lambda_0=-\lambda_1;\\
\frac{n_i(n_i+1)}{2}\|\la f\|^2+\frac{\lambda n_i(n_i-6)}{2}\|df\|^2- \lambda^2 n_i(n_i-4)\|f\|^2, \text{ if } \lambda_0=\lambda_1.
\end{cases}
\Eea
\end{lem}
\begin{proof} Note that in this case,
\Bea
\langle (\nabla\mathcal{R}ic)'(h), h \rangle_{L^2} &=& \langle ({\delta}^Dd^Dr)'(h) -(D^{*}Dr)'(h)-2(r\circ r)'(h),h\rangle_{L^2} + \frac{1}{2}\langle(\la s'(h))g +(|r|^2)'(h)g, h\rangle_{L^2}\\
&& +2\lambda^2\|h\|^2 
\Eea
In local coordinates $|r|^2$ can be written as $|r|^2=r_{ij}r_{kl}g^{ik}g^{jl}.$
Differentiating the expression we have,
\begin{eqnarray*}
(|r|^2)'(h)&=& r'(h)_{ij}r_{kl}g^{ik}g^{jl}+ r_{ij}r'(h)_{kl}g^{ik}g^{jl}+r_{ij}r_{kl}((g^{ik})'g^{jl}+g^{ik}(g^{jl})')\\
&=&2r'(h)_{ij}r_{ij}-r_{ij}r_{kj}h_{ik}-r_{ij}r_{il}h_{jl}\\
&=& n_i((\lambda_0+\lambda_1) \la f- 2\lambda_i^2f).
\end{eqnarray*}
Also, since $s$ is a constant and $\langle g, h\rangle =0,$ one obtains $\langle (\la s)'(h) g ,h\rangle = \langle \la s'(h)g,h\rangle.$ Hence, 
$$s'(h)= -(r,h)+\la tr h=n_i(\la^2 f-\lambda_i \la f)$$
Therefore $\langle (\la s)'(h)g, h\rangle_{L^2}= n_i^2(\|\la f\|^2- \lambda_i\|df\|^2).$ 
Now the result follows combining Lemmas \ref{lem1}, ~\ref{lem2}, ~\ref{lem3} and the above expressions.
\end{proof}
\subsection*{ Proof of Theorem \ref{main1} :} Let $|\lambda_i|=\lambda>0$ and $f_i\in C^{\infty}M_{i+1}$ for $i=0,1$ (indices $\mod 2$). We notice that $H(f_0g_1,f_1g_0)=0$. Therefore, to prove stability of $\mathcal{R}ic$ restricted to $\mathcal{M}_W$ it sufficient to prove that for each $i=0,1$ there exists $C_i$ such that 
$$H(f_ig_{i+1},f_ig_{i+1})\geq C_i \ \ \forall f_i\in C^{\infty}(M_i).$$
\begin{enumerate} 
\item
\textit{Case} $(\lambda_0=\lambda_1=\lambda) :$ If $\lambda_i>0$ then (\ref{eqn-BW}) implies that $\mu_i\geq \lambda $. Then the result follows from Lemma \ref{lem6} using (\ref{eqn-BW}). When  $\lambda_i<0$, consider the polynomial $p(a,x)=(a+1)x^2-(a-6)x-2(a-4)$. When $a=3$ or $4$, $p(a,x)>0$ for all $x>0$. When $a>4$ then
$$p(a,x)>0   \ {\rm for} \ \ x>c(a)=\frac{(a-6)+\sqrt{9a^2-36a+4}}{2(a+1)}.$$
Let $f_i$ be an eigenfunction corresponding to the first eigenvalue $\mu_i$ of $\la$. Lemma \ref{lem6} implies that $H(f_ig_{i+1},f_ig_{i+1})>0$ if and only if  $p(n_i,\frac{\mu_{i+1}}{\lambda})>0$. Therefore, if $\frac{\mu_{i+1}}{\lambda}>c(n_i)$ then $H(f_ig_{i+1},f_ig_{i+1})>0$. Hence the proof follows. 
\item
\textit{Case} $(\lambda_0=-\lambda_1) :$ $H(f_0g_1,f_0g_1)$ is positive since $\frac{\mu_0}{\lambda}>1.$ If $\frac{\mu_1}{\lambda} >  \frac{(n_1+2)+\sqrt{9n_1^2-20n_1-28}}{2(n_1+1)}$ then $H(f_1g_0,f_1g_0)$ is positive by similar argument as in the previous case.
\end{enumerate}
\subsection*{Proof of Theorem \ref{main2}} As a consequence of Theorem \ref{Ric:TT} and Theorem \ref{main1}, it is sufficient to prove stability of $\mathcal{R}ic$ restricted to conformal variations of $(g_0+g_1)$. Since $C^{\infty} M_0\oplus C^{\infty}M_1$ is dense in $C^{\infty}M$ we study $H$ restricted to $\{f_0g_0+f_1g_1 : f_i\in C^\infty(M_i), i=0,1\}$. Notice that $H(f_0g_0, f_1g_1)=0.$ As each $(M_i,g_i)$ is stable for $\mathcal{R}ic$, there exist positive constant $C_i>0$ such that for each $i=0,1$
$$H(fg_i,fg_i)\geq C_i \ \ \forall f\in C^{\infty}M_i \ {\rm with } \ \int_{M_i}f=0$$ 
Hence the theorem follows . 
\subsection*{Examples of Stable manifolds:} Some known examples of stable critical metrics of $\mathcal{R}ic$ are quotients of $S^n$ ($n\geq 3$), the complex projective space $\mathbb{C}P^m$, $m\geq 2$, compact hyperbolic manifolds $H^n$ with dimensions $3,4$, and compact hyperbolic manifolds $H^n$ ($n>4$) with first eigenvalue $\mu> \frac{2(n-1)(n-4)}{n}$ (see \cite{GV}). As a direct consequence of Theorem \ref{main2}, one observes that the following product manifolds are stable for $\mathcal{R}ic$.
\begin{enumerate}
\item $S^{n_1}\times S^{n_2}....\times S^{n_k}$ ($n_i\geq 3, \ i=1,2,..,k$)
\item $\mathbb{C}P^{n_1} \times \mathbb{C}P^{n_2}...\times \mathbb{C}P^{n_k}$ (for $n_i\geq 2, \ i=1,2,..,k$)
\item $H_1\times H_2\times...\times H_k$ where $H_i$ is a compact hyperbolic manifold with dimension $3$ or $4$ for $i=1,2,..,k.$ 
\item Product of Riemannian manifolds stated in 1,2,3.
\item Let $H_1^n, H_2^m$ ($n,m\geq 4$) be two compact hyperbolic manifolds  with first eigenvalues of Laplacian $\mu_1$ and $\mu_2$ respectively. The product manifold $H_1^n \times H_2^m$ is stable if and only if 
\Bea
&&\frac{\mu_1}{n-1}>\max \{ c(m), 2-\frac{8}{n} \} \text{ and} \ \frac{\mu_2}{m-1}>\max \{c(n),2-\frac{8}{m}\}
\Eea
i.e. in particular, if $\mu_1>2(n-1)$ and $\mu_2>2(m-1)$ then $H_1^n, \times H_2^m$ is stable. 
\item Let $H^m$ be a compact hyperbolic manifold. The product manifold $S^n\times H^m$ ($n,m\geq 3$, one of them bigger than $4$) is stable for $\mathcal{R}ic$ if and only if
$\frac{\mu}{m-1}>\max \left\{\frac{2(m-4)}{m}, \frac{(n+2)+\sqrt{9n^2-20n-28}}{2(n+1)}\right \}$, where $\mu$ is the first eigenvalue of the Laplacian on $H^m.$ 
\item $S^n\times H^m$ ($n,m\geq 3$) is stable if $\mu> 2(m-1)$.
\end{enumerate}  
\subsection*{Unstable manifolds :}
Let $H_1^n, H_2^m$ ($n,m\geq 4$) be compact hyperbolic space forms both stable for $\mathcal{R}ic$, then their Riemannian product $H_1^n \times H_2^m$ is unstable if 
\Bea
\text{either } 2- \frac{8}{n} <\frac{\mu_1}{n-1} < c(m) \ \ \text{ or } \ \ 2- \frac{8}{m}< \frac{\mu_2}{m-1}<c(n)
\Eea
where $\mu_i, c$ ($i=1,2$) are as in Theorem \ref{main1}.\\
Next, consider $M=S^n\times H^m$ such that the first eigenvalue $\mu$ of $\la$ of $H^m$ satisfies
$$\frac{2(m-4)}{m}<\frac{\mu}{m-1}<\frac{(n+2)+\sqrt{9n^2-20n-28}}{2(n+1)}.$$
Although $S^n$ and $H^m$ both are individually stable for $\mathcal{R}ic$, their product  $S^n\times H^m$ is unstable. For existence of hyperbolic manifolds with small eigenvalues we refer to \cite{RS}. 
%------------------------------------------------
\section{Second variations of other quadratic functionals}
\subsection{Stability of $\mathcal{F}_t$:}
In this section, we consider the Riemannian functional $\mathcal{F}_t$ and study its stability at the product of two closed Einstein manifolds. First, we describe the second variation of $\mathcal{F}_t$ at such critical points.
\begin{lem}\label{sv-F_t-general}
Let $(M^n,g)$ (where $n=n_0+n_1$) be the product of closed Einstein manifolds $(M_0^{n_0},g_0)$ and $(M_1^{n_1},g_1)$ ($n_0,n_1\geq 3$) with respective Einstein constants $\lambda_0$ and $\lambda_1$.
\begin{enumerate}
\item Let $h=\alpha_0\odot \alpha_1$ as in Lemma \ref{lem1}, one obtains
\Bea
&&\langle (\nabla\mathcal{F}_t)'(h), h \rangle_{L^2} =
\begin{cases}
\langle (\nabla\mathcal{R}ic)'(h), h \rangle_{L^2} ,  \text{ if }  \lambda_0 =-\lambda_1, \text{ and }n_0=n_1;\\
\| \la \alpha_0 \|^2 + \|\la \alpha_1 \|^2 +2 \langle \la \alpha_0, \alpha_0  \rangle_{L^2} \langle \la \alpha_1, \alpha_1 \rangle_{L^2} + 4\lambda^2\left( 2+tn \right)\\
\;\;\;\;\;\;\;\;\;\;\;\;\;\;\;-\lambda\left(5+2tn\right)\left(\langle \la \alpha_0, \alpha_0 \rangle_{L^2} +\langle \la \alpha_1, \alpha_1 \rangle_{L^2}\right);  \text{ if }  \lambda_0 =\lambda_1=\lambda.
\end{cases}
\Eea
\item For $h=fg_i$ as in Lemma \ref{lem1}, one obtains
\Bea
&&\langle (\nabla\mathcal{F}_t)'(h), h \rangle_{L^2}\\ 
&&=\begin{cases}
\frac{m((4t+1)m+1)}{2}\|\la f\|^2 - \lambda_i\frac{m((8t+1)m+2)}{2}\|df\|^2 + \lambda^2 m (m(2t-1)+4) \|f\|^2, \\
\; \;\; \;\; \;\; \;\; \;\; \;\; \;\; \;\; \;\; \;\; \;\; \; \; \;\; \;\; \;\; \;\; \;\; \;\;\; \;\; \;\; \;\; \;\; \;\; \;\; \;\; \;\; \;\; \;\; \;\; \; \; \;\; \;\; \;\; \;\; \;\; \;\;\;\; \;\; \; \; \;\;\text{ if } \lambda_0=-\lambda_1=\lambda, n_0=n_1=m;\\
\frac{n_i((4t+1)n_i+1)}{2}\|\la f\|^2+\lambda \frac{n_i}{2}(2t(n(n_i-1)-4n_i)+n_i-6)\|df\|^2
\\
\; \;\; \;\; \;\; \;\; \;\; \;\; \;\; \;\; \;\; \;\; \;\; \; \; \;\; \;\; \;\;  -\lambda^2n_i(t(n(n_i-2)-2n_i)+n_i-4)\|f\|^2,\text{                                                 if } \lambda_0=\lambda_1=\lambda.
\end{cases}
\Eea  
\end{enumerate}
\end{lem}
\begin{proof} Let $h=\alpha_0\odot \alpha_1$ such that $\|\alpha_0\|=1=\|\alpha_1\|$. Using (\ref{eqn-BW}) we have
\bea \label{scal1}
\nonumber \langle (\nabla\mathcal{S})'(h), h \rangle_{L^2} &=& \frac{2s^2}{n}\|h\|^2-2s \langle r'(h),h\rangle_{L^2} \\
\nonumber &=& \frac{4(n_0\lambda_0+n_1\lambda_1)^2}{n}-2(n_0\lambda_0+n_1\lambda_1)(\langle \la \alpha_0, \alpha_0 \rangle_{L^2} +\langle \la \alpha_1, \alpha_1 \rangle_{L^2})\\
&=&\begin{cases} 0, \; \; \; \; \; \; \; \; \; \; \; \; \; \; \; \; \; \; \; \; \; \; \; \; \; \; \; \; \; \; \; \; \; \; \; \; \; \; \; \; \; \; \; \; \; \; \; \;    {\rm if} \,\; \; \lambda_1=-\lambda_0,n_0=n_1;\\
2\lambda n(2\lambda -\langle\la \alpha_0, \alpha_0 \rangle_{L^2} -\langle \la \alpha_1, \alpha_1 \rangle_{L^2} ),\; {\rm if} \,\; \; \lambda_0=\lambda_1=\lambda.
\end{cases}
\eea
Consequently, using \eqref{scal1} and Theorem \ref{Ric:TT}, proof of (1) follows.\\
Let $h=fg_i$. Using similar computations as in Lemma \ref{lem6} for we have 
\Bea
&&\langle s'(h)r,h\rangle_{L^2} = n_i^2\lambda_i(\|df\|^2-\lambda_i\|f\|^2),\
\langle s.r'(h),h\rangle_{L^2} = sn_i\|df\|^2, \ \langle s.s'(h)g,h\rangle_{L^2} = sn_i^2(\|df\|^2-\lambda_i\|f\|^2)
\Eea
\noindent which implies that
\bea\label{scal2}
\nonumber &&\langle (\nabla\mathcal{S})'(h), h \rangle_{L^2} = \langle 2\la (s'(h))g -2s'(h)r-2s.r'(h)+s.s'(h)g, h\rangle_{L^2} +\frac{2n_is^2}{m+n}\|f\|^2\\
 &&=\begin{cases}
2m^2(\|\la f\|^2 - 2\lambda_i \|df\|^2 + \lambda^2  \|f\|^2), \text{ if } \lambda_0=-\lambda_1=\lambda \text{ and } n_0=n_1=m;\\
2n_i^2\|\la f\|^2+\lambda (nn_i(n_i-1)- 4n_i^2)\|df\|^2 + n_i\lambda^2 \left(2n_i-n (n_i-2)\right)\|f\|^2,\\ \,\,\,\,\,\,\,\,\,\,\,\,\,\,\, \,\,\,\,\,\,\,\,\,\,\,\,\,\,\,\,\,\,\,\,\,\,\,\,\,\,\,\,\,\,\,\,\,\,\,\,\,\,\,\,\,\,\,\,\, \,\,\,\,\,\,\,\,\,\,\,\,\,\,\,\,\,\,\,\,\,\,\,\,\,\,\,\,\,\,\,\,\,\,\,\,\,\,\,\,\,\,\,\,\, \,\,\,\,\,\,\,\,\,\,\,\,\,\,\,\,\,\,\,\,\,\,\,\,\,\,\,\,\,\,\,\,\,\,\,\,\,\,\,\,\,\,\,\,\ \text{    if } \lambda_0=\lambda_1=\lambda.
\end{cases}
\eea
Now (2) follows from combining \eqref{scal2} and Lemma \ref{lem6}.
\end{proof}
Note that if $\lambda_0=-\lambda_1$ then $(M,g)$ is a critical point of $\mathcal{F}_t$ for all $0\neq t\in \mathbb{R}$ if and only if $n_0=n_1$. The following result analyses the stability of $\mathcal{F}_t$ for such critical points.
\begin{thm} \label{st-F_t-opposite}
Let $(M^{2n},g)$ be a critical point of $\mathcal{F}_t$ which is the product of closed Einstein manifolds $(M_0^{n},g_0)$ and $(M_1^{n},g_1)$ ($n\geq 3$) with respective Einstein constants $\lambda>0$ and $-\lambda$. Let for each $i=0,1$, $\mu_i$ denote the first eigenvalue of the Laplacian on $(M_i,g_i).$ Then 
$(M,g)$ is stable for $\mathcal{F}_t$ if and only if each $(M_i,g_i)$ is stable for the same and one of the following conditions holds:
\Bea
&(1).& n=3, \text{ and } t >-\frac{1}{3};\\
&(2).&4 \leq n\leq 21 \text{ and }
\begin{cases} \frac{-(9n^2-20n-28)}{4n(8n-7)} >t >\frac{-(n+1)}{4n};\\
n= 4, \ t \geq \frac{-(9n^2-20n-28)}{4n(8n-7)} \text { with } 
\frac{\mu_{1}}{\lambda} > \frac{-b_{0}+ \sqrt{D}}{2((4t+1)n+1))};\\
n\geq 5, \begin{cases} -\frac{11}{16n} > t \geq \frac{-(9n^2-20n-28)}{4n(8n-7)} \text { with } \ \frac{\mu_{[i]}}{\lambda} > \frac{-b_{[i+1]}+ \sqrt{D}}{2((4t+1)n+1))};\\
t \geq -\frac{11}{16n} \ \text{ with }\frac{\mu_{1}}{\lambda} > \frac{-b_{0}+ \sqrt{D}}{2((4t+1)n+1))}.
\end{cases}
\end{cases}\\
&(2).&n\geq 22 \text{ and }
\begin{cases} t \geq -\frac{11}{16n}, \text{ with }  \frac{\mu_{1}}{\lambda} > \frac{-b_{0}+ \sqrt{D}}{2((4t+1)n+1))};\\ 
 -\frac{11}{16n} >t >\frac{-(n+1)}{4n}, \text{ with } \frac{\mu_{[i]}}{\lambda} > \frac{-b_{[i+1]}+ \sqrt{D}}{2((4t+1)n+1))}.
 \end{cases}
\Eea
where $b_0= -((8t+1)n+2)=-b_1, \ D= 4nt(8n-7)+ (9n^2-20n-28)$ and $[i]=i \mod 2$ .
\end{thm}
\begin{proof} 
For transverse-traceless variations of $\mathcal{F}_t$, observe that for $h=\alpha_0\odot \alpha_1$ as in Lemma \ref{sv-F_t-general}(1), it is immediate combining Lemma \ref{sv-F_t-general}(1) and Theorem \ref{Ric:TT}, that there is a constant $C>0$ such that 
$$\langle (\nabla\mathcal{F}_t)'(h), h \rangle_{L^2} \geq C\|h\|^2.$$
For conformal variations of $\mathcal{F}_t$, with $h=fg_i$ as in Lemma \ref{sv-F_t-general}(2), one obtains
\Bea
\langle (\nabla\mathcal{F}_t)'(h), h \rangle_{L^2}=\frac{n((4t+1)n+1)}{2}\|\la f\|^2 - \lambda_i\frac{n((8t+1)n+2)}{2}\|df\|^2 + n\lambda^2 (n(2t-1)+4) \|f\|^2.
\Eea
Considering the quadratic polynomial $p_i(t,x)= ax^2+ b_ix+c$ with 
$$a= (4t+1)n+1, \ b_0= -((8t+1)n+2)=-b_1 \ \text{ and } c=2(2t-1)n+8,$$
it follows that $\langle (\nabla\mathcal{F}_t)'(h), h \rangle_{L^2}>0$ if and only if $p_i(t,x)>0$ for $x \geq \frac{\mu_{i+1}}{\lambda}$ ($\lambda>0$). For $t\leq \frac{-(n+1)}{4n},$ $a\leq 0$, and $p_0(t,x)\leq -nx-(3n-7)<0$ for all $x\geq 0$ (as $n\geq 3$). On the other hand, for $t>\frac{-(n+1)}{4n},$ the discriminant of $p_i(t,x)$ given by $D= 4nt(8n-7)+ (9n^2-20n-28)<0$ whenever $t< \frac{-(9n^2-20n-28)}{4n(8n-7)}$ (i.e. both $p_i(t,x)>0$ as $a>0$) which is possible only when $n\leq 21$. 

For $t>\frac{-(n+1)}{4n},$ $D\geq 0$ if and only if $t\geq \frac{-(9n^2-20n-28)}{4n(8n-7)}$ and $p_i(t,x)>0$ if $x>\frac{-b_i+ \sqrt{D}}{2a}$. Since $\frac{\mu_{0}}{\lambda}>1$ and $\frac{\mu_{1}}{\lambda}\geq 0$, it follows immediately that the positivity conditions $\frac{\mu_{[i]}}{\lambda} > \frac{-b_{[i+1]}+ \sqrt{D}}{2a}$ are non-trivial only when $\frac{-b_{1}+ \sqrt{D}}{2a}>1$ and $\frac{-b_{0}+ \sqrt{D}}{2a}>0$ respectively. Note that for $n\in \mathbb{N}$,
\Bea
 -\frac{(9n^2-20n-28)}{4n(8n-7)}
 \begin{cases} > -\frac{(n+1)}{4n}, \ \text{ if } n\leq 21;\\
 < -\frac{(n+1)}{4n}, \text{ otherwise }.
 \end{cases}
\Eea  
Considering all possibilities, one can summarize conditions of positivity of $\langle (\nabla\mathcal{F}_t)'(h), h \rangle_{L^2}$ as follows.
\begin{itemize}
\item $n=3, \text{ with } t>-\frac{1}{3}$: Clearly, for $t\geq 0$,
\Bea
p_0(t,x)&=& \frac{1}{2}((4x^2-5x+2)+12t(x-1)^2))>0, \text{ if }x>1, \text{ and}\\
p_1(t,x)&=& \frac{1}{2}((4x^2+5x+2)+12t(x+1)^2))>0, \text{ if }x\geq 0.
\Eea
Also, for $(-\frac{(9n^2-20n-28)}{4n(8n-7)}=) \frac{7}{204}>t>-\frac{1}{3}$, $D<0$ 
and $p_i(t,x)>0$ for each $i=0,1$. Hence in this case, $\langle (\nabla\mathcal{F}
_t)'(h), h \rangle_{L^2}>0$ whenever $t>-\frac{1}{3}.$ \\
 
\item $n= 4, \text{ with }t \geq -\frac{(9n^2-20n-28)}{4n(8n-7)} \big(>-\frac{(n+1)}{4n}\big)$ and $n \geq 5,\text{ with }t \geq -\frac{11}{16n}$:\\
Since in both the cases, $\frac{-b_{1}+ \sqrt{D}}{2a} \leq 1$ and $\frac{-b_{0}+ \sqrt{D}}{2a}>0$, the only non-trivial condition for positivity is $\frac{\mu_1}{\lambda} > \frac{-b_0+ \sqrt{D}}{2a}$.\\
 
\item $n\geq 5,$ together with $-\frac{11}{16n}>t \geq \text{max} \{-\frac{(n+1)}{4n}+\epsilon, -\frac{(9n^2-20n-28)}{4n(8n-7)}\}$ for some $\epsilon>0$: \\
In this case, both conditions for positivity are non-trivial (as $\frac{-b_{1}+ \sqrt{D}}{2a} > 1$ and $\frac{-b_{0}+ \sqrt{D}}{2a}>0$). \\
 
\item $n\geq 22,$ and $-\frac{11}{16n}>t > -\frac{(n+1)}{4n}$: Here also both the conditions for positivity are non-trivial. 
\end{itemize}

Combining all these possibilities above, it follows that $\langle (\nabla\mathcal{F}_t)'(h), h \rangle_{L^2}>0$ for $h=fg_i$, if and only if one of the conditions in the hypothesis holds. The final consequence is immediate following similar arguments as in Theorem \ref{main2}.
\end{proof}

\begin{rem}
Combining Theorem \ref{st-F_t-opposite} with Theorem 1.7 and Theorem 1.9 from \cite{GV}, it follows that a critical point $(M^{6},g)$ which is the product of closed Einstein manifolds $(M_0^{3},g_0)$ and $(M_1^{3},g_1)$ with Einstein constants of opposite signs, is stable for $\mathcal{F}_t$ provided $\frac{1}{3}>t>-\frac{1}{3}.$
Moreover, a critical point $(M^{2n},g)$ ($n\geq 4$) as the product of closed Einstein stable manifolds $(M_0^{n},g_0)$ and $(M_1^{n},g_1)$ with Einstein constants of opposite signs, is unstable whenever $t\leq -\frac{(n+1)}{4n}$ or $n\geq 22$ with $t< -\frac{(9n^2-20n-28)}{4n(8n-7)}$. On the other hand, a critical point $(M^{2n},g)$ ($n\geq 4$) as above is stable if both $(M_0^{n},g_0)$ and $(M_1^{n},g_1)$ are stable,  $t\geq -\frac{11}{16n}$ and the first eigen value of the Laplacian of $(M_1,g_1)$ (with $\lambda_1<0$) is sufficiently large.
\end{rem}

\subsection{Stability of $\mathcal{W}_2$}
Let $(M^n,g)$ be the product of a spherical space form $(M^{n_0}_0,g_0)$ and a compact hyperbolic space form $(M^{n_1}_1,g_1)$. Then $g$ is conformally flat. Hence $g$ is a global minima for $\mathcal{W}_2.$ Thus it would be interesting to see if $g$ is stable for $\mathcal{W}_2$. 
\begin{lem}\label{lW}
Let $(M^n,g)$ $(n=n_0+n_1)$ be the product of two compact space forms $(M^{n_0}_0,g_0)$ and $(M^{n_1}_1,g_1)$ such that one of the following conditions holds:
\begin{enumerate} 
\item $n_0,n_1\geq 3$ with $K_{g_0}=1=-K_{g_1}$,
\item $n_0\geq 3$, $n_{1}=1$, $K_{g_0}=\pm 1$, 
\end{enumerate}
(where $K_g$ denotes the sectional curvature of $g$). 
Then there exists $C>0$ such that for any TT-tensor $h$ on $M$,
$$\langle (\nabla\mathcal{W}_2)'(h), h \rangle_{L^2} \geq C\|h\|^2.$$
\end{lem}

\begin{proof}
Following Theorem \ref{Ric:TT}, it suffices to assume that $h=\alpha_0\odot \alpha_1$ as above.

Equation \ref{weyl2} implies that
\Bea
\langle (\nabla \mathcal{W}_2)'(h),h\rangle_{L^2}=\langle \nabla\mathcal{R}'(h)- \frac{4}{n-2}\{(\nabla\mathcal{R}ic)'(h)- \frac{1}{2(n-1)}(\nabla\mathcal{S})'(h)\},h\rangle_{L^2}.
\Eea 
Putting $\lambda_0=n_0-1$ and $\lambda_1=-(n_1-1)$ in Lemmas \ref{lem1},~\ref{lem2} and \ref{lem3}, one has
\Bea 
&& \langle (\nabla\mathcal{R}ic)'(h), h \rangle_{L^2} =\| \la \alpha_0 \|^2 + \|\la \alpha_1 \|^2 +2 \langle \la \alpha_0, \alpha_0  \rangle \langle \la \alpha_1, \alpha_1 \rangle - 4(n_0-1)(n_1-1)\\
&& +\frac{4}{n}((n_0(n_0-1)^2+n_1(n_1-1)^2))
-(3n_0-2n_1-1) \langle \la \alpha_0, \alpha_0 \rangle -(3n_0-2n_1+1)\langle \la \alpha_1, \alpha_1 \rangle.  
\Eea

Next, using \eqref{scal1}, it follows that 
\Bea
\langle (\nabla\mathcal{S})'(h), h \rangle_{L^2} =\frac{4}{n}(n_0-n_1)^2(n-1)^2-2(n_0-n_1)(n-1)(\langle \la \alpha_0, \alpha_0 \rangle_{L^2} +\langle \la \alpha_1, \alpha_1 \rangle_{L^2})
\Eea

Using Lemmas \ref{lem1} and \ref{lem4}, one has
\Bea
&&\langle (\nabla\mathcal{R})'(h), h \rangle_{L^2} = 4(\| \la \alpha_0 \|^2 + \|\la \alpha_1 \|^2 +2 \langle \la \alpha_0, \alpha_0  \rangle_{L^2} \langle \la \alpha_1, \alpha_1 \rangle_{L^2})-8(n_0-1)(n_1-1)+8(n-2)\\
&& +(5(n_1-n_0)-4)\langle \la \alpha_0, \alpha_0  \rangle_{L^2} + (5(n_1-n_0)+4)\langle \la \alpha_1, \alpha_1  \rangle_{L^2} + \frac{8}{n}(n_0(n_0-1)+n_1(n_1-1))
\Eea

Combining above equations, one obtains
\Bea
&&\langle (\nabla \mathcal{W}_2)'(h),h\rangle_{L^2}=
\left(\frac{4(n-3)}{n-2}\right)\left(\| \la \alpha_0 \|^2 + \|\la \alpha_1 \|^2 +2 \langle \la \alpha_0, \alpha_0 \rangle_{L^2} \langle \la \alpha_1, \alpha_1 \rangle_{L^2}\right)\\
&&+ \left(5(n_0-n_1)+\frac{4(n_0-2n_1+1)}{n-2}\right)\langle \la \alpha_0, \alpha_0\rangle_{L^2} +\left(5(n_1-n_0)+\frac{4(2n_0-n_1-1)}{n-2}\right)\langle \la \alpha_1, \alpha_1\rangle_{L^2}\\
&&+ \frac{8(n_1-n_0)^2(n-1)}{n(n-2)}+8(n-2)+ \frac{8(n_0(n_0-1)+n_1(n_1-1))}{n}\\
&&-\frac{8(n_1-1)(n_0-1)(n-4)}{n-2}-\frac{16(n_0(n_0-1)^2+n_1(n_1-1)^2)}{n(n-2)}
\Eea

Setting $n_1=1$, it is easy to see from the above inequality  that 
\Bea
\langle (\nabla \mathcal{W}_2)'(h),h\rangle_{L^2} >0.
\Eea

When both $n_i\geq 3,$ using similar arguments as in Lemma \ref{R:TT}, it follows that $\langle \la \alpha_0, \alpha_0  \rangle_{L^2} \geq 2(n_0-1)$, $\|\la \alpha_0\|^2 \geq 4(n_0-1)^2$ and the above equation reduces to
\bea
&&\langle (\nabla \mathcal{W}_2)'(h),h\rangle_{L^2} \geq 
\frac{16(n-3)(n_0-1)^2}{n-2}+10(n_1-n_0)(n_0-1)+\frac{8(n_0-2n_1+1)(n_0-1)}{n-2}\\ 
\nonumber && -\frac{8(n_0-1)(n_1-1)(n-4)}{n-2}+ \left(\frac{16(n-3)(n_0-1)}{n-2}+5(n_1-n_0)+\frac{4(2n_0-n_1-1)}{n-2}\right)\langle \la_H \alpha_1, \alpha_1 \rangle_{L^2} \\
&&\nonumber + \frac{8(n_1-n_0)^2(n-1)}{n(n-2)}+8(n-2)+ \frac{8(n_0(n_0-1)+n_1(n_1-1))}{n}
-\frac{16(n_0(n_0-1)^2+n_1(n_1-1)^2)}{n(n-2)} 
\eea
Consequently, it follows that
\Bea
&&\langle (\nabla \mathcal{W}_2)'(h),h\rangle_{L^2} >\frac{(6n-20)(n_0-1)^2}{n-2}+2(n_1-1)(n_0-1)\\ 
&&+ \frac{1}{n-2}\left(n_0(11n_0-6)+16(n_0-1)(n_1-3)+n_1(5n_1-14)\right)\langle \la_H \alpha_1, \alpha_1 \rangle_{L^2} >0. 
\Eea

\end{proof}
This leads to the following result.
\begin{thm} 
Let $(M,g)$ be the product of two closed non-Euclidean space forms whose sectional curvatures have absolute value $1$. Then $(M,g)$ is a stable critical point of the Riemannian functional $\mathcal{W}_{\frac{n}{2}}$ defined by 
$$\mathcal{W}_{\frac{n}{2}}(g)= \int_M |W_g|^{\frac{n}{2}} dv_g. $$
\end{thm}  
\begin{proof}
Since $\mathcal{W}_{\frac{n}{2}}$ is invariant under conformal deformations, we only need to check the positivity of $\langle (\nabla\mathcal{W}_{\frac{n}{2}})'(h),h\rangle_{L^2}$ for all TT-tensors $h.$ Also, it follows that (see Lemma 4 in \cite{SM}, for a proof) at a critical point $g$ of both $\mathcal{W}$ and $\mathcal{W}_{\frac{n}{2}}$, 
$$\langle (\nabla\mathcal{W}_{\frac{n}{2}})'(h),h\rangle_{L^2}= \|W\|^{\frac{n-4}{2}}\langle (\nabla\mathcal{W}_2)'(h),h\rangle_{L^2}, \ forall \ h\in tr^{-1}(0)\cap \dg^{-1}(0).$$
Since conformally flat metrics are global minima for both $\mathcal{W}_2$ and $\mathcal{W}_{\frac{n}{2}}$ the result follows from Lemma \ref{lW}. 
\end{proof}
\subsection{Stability of $\mathcal{R}$}
\begin{cor}\label{R:TT} 
Let $(M^{2n},g)$ ($n\geq 3$) be the product of a spherical space form $(M_0^n,g_0)$ and a compact hyperbolic space form $(M_1^n,g_1)$. Then there exists a constant $C>0$ such that for any TT-tensor $h$ on $M$,
\be \label{eqn:R-TT}
\langle (\nabla\mathcal{R})'(h), h \rangle_{L^2} \geq C\|h\|^2 
\ee
\end{cor}
\begin{proof}
Since on $(M^{2n},g)$, one can write (following \eqref{Weyl1}) 
$$\mathcal{R}=\mathcal{W}_2+  \frac{2}{n-1}\mathcal{F}_{\frac{-1}{2(2n-1)}},$$
the result is immediate from Theorem \ref{st-F_t-opposite} and Lemma \ref{lW}.
\end{proof}

\begin{cor}\label{lR}
Let $(M^{2n},g)$ be as in Corollary \ref{R:TT} and let $h=fg_i$ as in Lemma \ref{lem4}. Then  
\Bea 
\langle (\nabla \mathcal{R})'(h),h\rangle >0 \text{ if and only if} 
\begin{cases}
n \leq 4, \\
n\geq 5 \text{ and } \mu_i>\sqrt{(n-1)(n-4)}.
\end{cases}
\Eea
where $\mu_i$ is the first eigenvalue of the Laplacian of $(M_i,g_i)$ for $i=0,1$.
\end{cor} 
\begin{proof}
It follows from a combination of Lemmas \ref{lem1} and \ref{lem4}, that
\Bea 
\langle (\nabla \mathcal{R})'(h),h\rangle_{L^2} &=& 2\langle ({\delta}^Dd^Dr)'(fg_i)-\langle (\Rt)'(fg_i),\frac{1}{2} (|R|^2)'(fg_i), fg_i \rangle_{L^2} + 4n(n-1)\|f\|^2\\
&=& 2n\|\la f\|^2 -2n(n-1)(n-4)\|f\|^2
\Eea
Now, one observes that 
\Bea 
\langle (\nabla \mathcal{R})'(h),h\rangle_{L^2} >0 \text{ if and only if }\mu_i^2-(n-1)(n-4)>0
\Eea
and the estimate follows.
\end{proof}
Consequently we have the following theorem using \cite{BM} and Corollaries \ref{R:TT}, \ref{lR} above.
\begin{thm}\label{stabilityR} Let $(M,g)$ be the product of $S^n$ and a compact hyperbolic manifold $H^n$. Then $(M,g)$ is stable for $\mathcal{R}$ if and only if one of the following conditions holds.
\Bea
(i) \ n\in \{3, 4\} \, \  \ (ii) \ n\geq 5 \text{ and }\mu>\sqrt{(n-1)(n-4)}.
\Eea
where $\mu$ is the first eigenvalue of the Laplacian of $H^n$.
\end{thm}
Recently, stability of rank 1 symmetric spaces as critical points of $\mathcal{R}$ and $\mathcal{W}_{\frac{n}{2}}$ has been established in \cite{SM}. The behaviour of irreducible symmetric spaces of higher rank as critical metrics of $\mathcal{R}$ is not completely understood, but some unstable critical points of the functional $\mathcal{R}_{\frac{n}{2}}$ have been pointed out in \cite{BM}. 
The above theorem implies that if the first eigenvalue of the laplacian of a hyperbolic manifold $H^n$ is sufficiently small and $n>5$ then its product with $S^n$ is unstable for $\mathcal{R}.$

\end{document}